\RequirePackage{fix-cm}
\documentclass[smallextended]{svjour3}       
\smartqed  
\usepackage{graphicx}
\usepackage{amsmath}
\usepackage{verbatim}
\usepackage{amssymb}
\usepackage{upref}
\usepackage{enumerate}
\usepackage{lineno,hyperref}
\usepackage{color}
\newtheorem{thm}{Theorem}
 
 \newtheorem{lem}[thm]{Lemma}

 \newtheorem{defn}[thm]{Definition}
 \newtheorem{rem}[thm]{Remark}
 \newtheorem{ex}{Example}

\newcommand{\R}{{\mathbb R}}

\newcommand{\N}{{\mathbb N}}

\begin{document}

\title{Critical regularity of nonlinearities in semilinear classical damped wave equations 
}


\author{M. R. Ebert         \and G. Girardi
        and M. Reissig 
}


\institute{M. R. Ebert \at
              Departamento de Computa\c{c}\~{a}o e Matem\'atica,
 Universidade de S\~{a}o Paulo (USP),  FFCLRP \\
  Av. dos Bandeirantes, 3900, CEP 14040-901, Ribeir\~{a}o Preto(SP), Brazil \\
              \email{ebert@ffclrp.usp.br}           
           \and
           G. Girardi \at
            Dipartimento di Matematica, Universit\`{a} degli Studi di Bari Aldo Moro, Via Orabona 4, Bari, Italy\\
            \email{giovanni.girardi@uniba.it}
              \and
              M. Reissig \at
              Faculty for Mathematics and Computer Science
 Technical University Bergakademie Freiberg
  Pr\"uferstr. 9, 09596 Freiberg, Germany\\
 \email{ reissig@math.tu-freiberg.de}
}

\date{Received: date / Accepted: date}

\maketitle

\begin{abstract}
In this paper we consider the Cauchy problem for the semilinear damped wave equation
\[ u_{tt} -  \Delta u + u_t= h(u), \qquad u(0,x)=\phi(x), \qquad u_t(0,x)= \psi(x),\]
where $h(s)=|s|^{1+ \frac2{n}}\mu(|s|)$. Here $n$ is the space dimension and $\mu$ is a modulus of continuity.
Our  goal is to obtain sharp conditions on $\mu$ to obtain
a threshold between global (in time) existence of small data solutions (stability of the zero solution) and blow-up behavior even of small data solutions.
\end{abstract}

\section{Introduction}
\label{intro}

In ~\cite{TY}, the authors proved the global existence of small data energy solutions for the semilinear damped wave equation
\begin{equation}\label{eq:dampedwave}
 u_{tt}-\Delta u + u_t = |u|^p, \qquad u(0,x)=\phi(x), \qquad u_t(0,x)= \psi(x),\end{equation}
in the supercritical range~$p>1+2/n$, by assuming  compactly
supported small data from the energy space. The compact support assumption on the data can be removed. By only assuming data in Sobolev spaces, a global (in time) existence result was proved in space dimensions~$n=1,2$ in~\cite{IMN04}, by using energy methods, and in space dimension~$n\leq5$ in~\cite{N04}, by using~$L^r-L^q$ estimates, $1\leq r\leq q\leq \infty$.
Nonexistence of general global (in time) small data solutions is proved in~\cite{TY} for~$1<p< 1+2/n$ and in \cite{Z} for $p= 1+2/n$ . The exponent~$1+2/n$ is well known as Fujita exponent and it is the critical power for the following semilinear parabolic Cauchy problem (see~\cite{F66}):
\begin{equation} \label{CPparabolic}
v_t -\triangle v = v^p\,, \qquad v(0,x)=v_0(x)\geq0.
\end{equation}
The diffusion phenomenon between linear heat and linear classical damped wave
models (see \cite{HM},  \cite{MN03},  \cite{N04} and \cite{N03}) explains the parabolic character of classical damped wave
models with power nonlinearities from the point of decay estimates of solutions.\\
In the mathematical literature (see for instance \cite{ER}) the situation is in general described as follows: We have a semilinear Cauchy problem
\[ L(\partial_t,\partial_x,t,x)u=|u|^p, \,\,\,u(0,x)=\phi(x),\,\,\,u_t(0,x)=\psi(x),\]
where $L$ is a linear partial differential operator. Then the authors would like to find a critical exponent $p_{crit}$ in the scale $\{|u|^p\}_{p>0}$, a threshold between two different qualitative behaviors of solutions. As examples see the models (\ref{eq:dampedwave}) or (\ref{CPparabolic}). \smallskip

{\it The main concern of this paper is to show by the aid of the model (\ref{eq:dampedwave}) that the restriction to the scale $\{|u|^p\}_{p>0}$ is too rough to verify the critical non-linearity
or the critical regularity of the non-linear right-hand side.} \smallskip

For this reason we turn to the  Cauchy problem  for the semilinear damped wave equation
\begin{equation}
\label{eq:CPnonlin}
u_{tt}-\Delta u+u_t=h(u), \,\,\, u(0,x)=\phi(x), \,\,\, u_t(0,x)= \psi(x),
\end{equation}
in $[0,\infty)\times \R^n$, where $h(s)= |s|^{1+\frac{2}{n}}\mu(|s|)$. Here $\mu=\mu(s),\,s \in [0,\infty)$, is a modulus of continuity, which provides an additional regularity of the right-hand side
$h=h(s)$ for $s \in [0,\infty)$.
\begin{defn} \label{Defmodulusofcontinuity}
A function $\mu: [0,\infty)\to [0,\infty)$ is called a modulus of continuity, if $\mu$ is a continuous, concave and increasing function satisfying $\mu(0)=0$.
\end{defn}
Our goal is to discuss the influence  of the function $\mu$ on the global (in time) existence of small data Sobolev solutions or on statements for blow-up of Sobolev solutions to \eqref{eq:CPnonlin}.
In the following result, we assume that the  modulus of continuity $\mu$ given in \eqref{eq:CPnonlin} satisfies the following two conditions:
\begin{equation}
\label{eq:conditionsmu}
 s^k|\mu^{(k)}(s)|\leq C\mu(s)\,\,\,\mbox{for}\,\,\,1\leq k\leq n,  \ s \in (0,s_0],\quad \text{ and } \quad \int_{C_0}^\infty \frac{\mu(s^{-1})}{s} ds <\infty,
\end{equation}
where $C$ and $C_0$ are sufficiently large positive constants and $s_0$ is a sufficiently small positive constant.
\begin{rem} \label{Remmodulusofcontinuity}
In the further considerations we need a suitable modulus of continuity satisfying the conditions (\ref{eq:conditionsmu}) on a small interval $[0,s_0]$ only.
Nevertheless we can assume that the modulus of continuity can be continued to the real line in such a way that the properties from Definition \ref{Defmodulusofcontinuity} are satisfied.
\end{rem}
\begin{thm}\label{GEDW}
Let $n=1,2$ and \[ (\phi,\psi)\in \mathcal{A}:=(H^{1+ \lfloor \frac{n}{2}\rfloor}\cap L^1)\times (H^{\lfloor \frac{n}{2}\rfloor }\cap L^1),\] where we denote by $\lfloor \cdot\rfloor$ the floor function. Assume \eqref{eq:conditionsmu}. Then, the following statement holds for a sufficiently small $\varepsilon_0>0$: if
\[ \|(\phi,\psi)\|_{\mathcal{A}} \leq \varepsilon\,\,\,\mbox{for}\,\,\,\varepsilon \leq \varepsilon_0, \]
then there exists a unique globally (in time) Sobolev solution $u$ to \eqref{eq:CPnonlin} belonging to the function space
\[ C\big([0,\infty), H^{1}(\R^n)\cap L^{\infty}(\R^n)\big), \]
such that the following decay estimates are satisfied:
\begin{align*}
\| u(t,\cdot)\|_{L^{\infty}}&\leq C (1+t)^{-\frac{n}{2}}\|(\phi,\psi)\|_{\mathcal{A}},\\
\|  \nabla_x^k u(t,\cdot)\|_{L^2}&\leq C (1+t)^{-\frac{n+2k}{4}}\|(\phi,\psi)\|_{\mathcal{A}}, \qquad k=0,1.
\end{align*}
\end{thm}
\begin{rem}
The key tool to prove Theorem \ref{GEDW} is to apply  estimates for solutions to the  parameter-dependent Cauchy problem for the linear
classical damped wave equation (Lemma \ref{th:Linear}).
By using more general ~$L^r-L^q$ estimates, $1\leq r\leq q\leq \infty$, derived in \cite{N04} for the linear damped wave equation, one can also obtain a global (in time) existence result for higher dimensions $n$, but this aim is beyond the scope of this paper.
\end{rem}
\begin{ex} The hypotheses of Theorem \ref{GEDW} hold for the following functions $\mu$ (see also Remark \ref{Remmodulusofcontinuity}):
\begin{enumerate}
\item $\mu(s)=s^p, p>0$;
\item  $\mu(s)=(\log (1+s))^p, p>0$;
\item  $\mu(0)=0$  and
$\mu(s)=\Big(\log \frac{1}{s}\Big)^{-p}, p>1$;
\item $\mu(0)=0$  and $\mu(s)=\Big(\log \frac{1}{s}\Big)^{-1}\Big(\log\log\frac{1}{s}\Big)^{-1}\cdots \Big(\log^{k} \frac{1}{s}\Big)^{-p},\,\, \ p>1,\,\,\, k \in \N$.
\end{enumerate}
 \end{ex}
The next result shows that the integral condition on the function $\mu$ in \eqref{eq:conditionsmu} can not be relaxed.
\begin{thm}\label{resultblowupdampedwave} Consider the Cauchy problem
\begin{equation}\label{eq:blowupdampedwave}
\begin{cases}
u_{tt} - \Delta  u + u_t=|u|^{1+\frac2n}\mu(|u|),
  \quad (t,x)\in (0,\infty)\times \R^n,
\\
(u(0,x),u_t(0,x))=(0,g(x)),
  \quad x\in \R^n.
\end{cases}
\end{equation}
Here $\mu=\mu(s),\,s \in [0,\infty)$ is a modulus of continuity which satisfies the condition
\begin{equation}
\label{eq:blowupcondition}
\int_{C_0}^\infty \frac{\mu(s^{-1})}{s} \,ds =\infty,
\end{equation}
where $C_0$ is a sufficiently large positive constant. Moreover, we assume that the function $h: s \in \mathbb{R} \to h(s):=s^{1+\frac2n}\mu(s)$ is convex on $\mathbb{R}$. Suppose that the data
\[ g\in \mathcal{A}:=H^{[n/2]}(\R^n)\cap L^1(\R^n),\] such that
$$ \int_{\R^n} g(x)\,dx>0.$$
Then, in general we have no global (in time) existence of Sobolev solutions even if the data are supposed to be very small in the following sense:
  \[ \|g\|_{\mathcal{A}}\leq \varepsilon\,\,\,\mbox{for}\,\,\, \varepsilon\leq \varepsilon_0.\]
  \end{thm}
To prove Theorem \ref{resultblowupdampedwave} we will follow the approach used in \cite{IS19} in which the authors get a sharp upper bound for the lifespan of solutions to some critical semilinear parabolic, dispersive and hyperbolic equations, by using a test function method.
  \begin{ex} The hypotheses of Theorem \ref{resultblowupdampedwave} hold for the following functions $\mu$ (see also Remark \ref{Remmodulusofcontinuity}):
\begin{enumerate}
\item  $\mu(0)=0$  and
$\mu(s)=\Big(\log \frac{1}{s}\Big)^{-p}, 0 < p \leq 1$;
\item $\mu(0)=0$  and $\mu(s)=\Big(\log \frac{1}{s}\Big)^{-1}\Big(\log\log\frac{1}{s}\Big)^{-1}\cdots \Big(\log^{k} \frac{1}{s}\Big)^{-p}, \, p\in (0,1], \ k \in \N$.
\end{enumerate}
 \end{ex}
\begin{rem} \label{Remconvexity}
Let us discuss the assumption in Theorem \ref{resultblowupdampedwave} that the function \[ h: s \in \mathbb{R} \to h(s):=s^{1+\frac2n}\mu(s)\,\,\,\mbox{ is convex on}\,\,\, \mathbb{R}.\]
In a small right-sided neighborhood of $s=0$, this hypothesis can be replaced by
the condition
\[ s^k\mu^{(k)}(s)=o(\mu(s))\,\,\,\mbox{for}\,\,\, s \to +0,\,\,\,k=1,2. \]
Indeed, it is sufficient  to verify that on a small interval $(0,s_0]$
\[ h^{''}(s)= s^{\frac2{n}-1}\left( \frac2{n}(1+\frac2{n})\mu(s) + 2(1+\frac2{n})s\mu'(s) + s^2\mu^{''}(s)\right)\geq 0.\]
This condition is satisfied in our examples. Outside this interval we can choose a convex continuation of $h$.
\end{rem}

\section{Global existence of small data solutions} \label{Secglobalexistencesemilineardampedwave1d2d}

In the proof of Theorem \ref{GEDW} we are going to use the following estimates for Sobolev solutions to the parameter-dependent Cauchy problem for the linear classical damped wave equation.
\begin{lem}[Lemma 1 in \cite{Mat}]
\label{th:Linear}
Let \[ (\phi,\psi)\in \mathcal{A}:=(H^{1+ \lfloor \frac{n}{2}\rfloor}\cap L^1)\times (H^{\lfloor \frac{n}{2}\rfloor }\cap L^1).\] Then, the Sobolev solutions to the Cauchy problem
\begin{equation}
\label{eq:CPlin}
u_{tt}-\Delta u+u_t=0, \,\,\, u(s,x)=\phi_s(x), \,\,\, u_t(s,x)= \psi_s(x),
\end{equation}
satisfies the following estimates for $t\geq 0$:
\[
\| u(t,\cdot)\|_{L^{\infty}}\leq C (1+t-s)^{-\frac{n}{2}}\left( \|\phi_s\|_{L^1} + \|\phi_s\|_{H^{1+\lfloor \frac{n}{2}\rfloor}} +\|\psi_s\|_{L^1}+ \|\psi_s\|_{H^{\lfloor \frac{n}{2}\rfloor}} \right),\]
and for $ k=0,1, 1+\lfloor \frac{n}{2}\rfloor$
\[\| \nabla_x^k u(t,\cdot)\|_{L^2}\leq C (1+t-s)^{-\frac{n+2k}{4}}\left( \|\phi_s\|_{L^1} + \|\phi_s\|_{H^k} +\|\psi_s\|_{L^1}+ \|\psi_s\|_{H^{k-1}} \right).
\]
\end{lem}
\begin{proof}[ Theorem \ref{GEDW} ] \\
The space of Sobolev solutions is $X(t)=C\big([0,t], H^1(\R^n)\cap L^{\infty}(\R^n)\big)$. Taking into consideration the estimates of Lemma \ref{th:Linear} we define on $X(t)$ the norm
\[\|u\|_{X(t)}=\sup_{\tau\in [0,t]}\Big\{\sum_{k=0}^1(1+\tau)^{\frac{n+2k}{4}} \|\nabla^k u(\tau,\cdot)\|_{L^2}+ (1+\tau)^{\frac{n}{2}}\|u(\tau,\cdot)\|_{L^{\infty}}\Big\}. \]
For arbitrarily given data $(\phi,\psi)\in \mathcal{A}$ we introduce the operator
\[ N: u\in X(t)\to u^{lin}+\int_0^t \Phi(t,s,\cdot)\ast_{(x)} h(u(s,\cdot))(x)\, ds \]
in~$X(t)$, where by $u^{lin}$ we denote the solution to the linear parameter-dependent Cauchy problem \eqref{eq:CPlin} with initial data $(\phi,\psi)$. By \[ \Phi(t,s,\cdot)\ast_{(x)} h(u(s,\cdot))(x) \] we denote the Sobolev solution to the Cauchy problem~\eqref{eq:CPlin} with~$\phi_s\equiv 0$ and~$\psi_s=h(u(s,\cdot))$.
We will prove that
\begin{eqnarray}
\label{eq:propN1}
& \|Nu\|_{X(t)}\leq
     C_0\|(\phi,\psi)\|_{\mathcal{A}}+ \tilde{C}_{\varepsilon_0}\|u\|^{1+\frac{2}{n}}_{X(t)},\\
\label{eq:contractionN}
& \|Nu-Nv\|_{X(t)}
     \leq C_{\varepsilon_0}\|u-v\|_{X(t)}\big(\|u\|^{\frac{2}{n}}_{X(t)}+ \|v\|^{\frac{2}{n}}_{X(t)}\big),
\end{eqnarray}
where $C_{\varepsilon_0}$ and $\tilde{C}_{\varepsilon_0}$ tend to $0$ for $\varepsilon_0$ to $0$.\\
First of all we have after applying Lemma \ref{th:Linear} for all $t>0$ the estimate
\begin{eqnarray} \label{Linearestimate} \|u^{lin}\|_{X(t)} \leq C_0\|(\phi,\psi)\|_{\mathcal{A}},\end{eqnarray}
where the constant $C_0$ is independent of $t$. Consequently, it remains to estimate
\[ G(u)(t,x):=\int_0^t \Phi(t,s,x)\ast_{(x)} h(u(s,x)) \,ds.\]
For $j=0,1$ we have
\[ \|\nabla^j G(u)(t,\cdot)\|_{L^2}\leq \int_0^t (1+t-s)^{-\frac{n}{4}-\frac{j}{2}}\|h(u(s,\cdot))\|_{L^1\cap L^2}ds. \]
It holds
\begin{equation*}
\|h(u(s,\cdot))\|_{L^1\cap L^2}\leq \mu(\|u(s,\cdot)\|_{L^{\infty}})\, \||u(s,\cdot)|^{1+\frac2{n}}\|_{L^1\cap L^2}.
\end{equation*}
Thus, by using that \[ \|u(s,\cdot)\|_{L^{\infty}}\leq (1+s)^{-\frac{n}{2}}\|u\|_{X(s)} \] and the monotonicity of $\mu=\mu(s)$ we get the following estimate:
\begin{equation}
\label{eq:norminfty}
\mu(\|u(s,\cdot)\|_{L^{\infty}})\leq   \mu\big((1+s)^{-\frac{n}{2}}\|u\|_{X(s)}\big).
\end{equation}
Let us assume $\|u\|_{X(t)}\leq \varepsilon_0$ for all $t>0$ and some $\varepsilon_0>0$ sufficiently small. Then, since the norm in $X(t)$ is increasing with respect to $t$, we can estimate the right-hand side of \eqref{eq:norminfty} by
\[ \mu\big(\varepsilon_0(1+s)^{-\frac{n}{2}}\big). \]
Moreover, to estimate $\||u(s,\cdot)|^{1+\frac2{n}}\|_{L^1\cap L^2}$ we may apply the Gagliardo-Nirenberg inequality and obtain
\begin{equation}
\label{eq:L3norm}
\|u(s,\cdot)\|_{L^{1+\frac2{n}}}^{1+\frac2{n}}\leq C\|\nabla u(s,\cdot)\|_{L^2}^{1-\frac{n}{2}}\|u(s,\cdot)\|_{L^2}^{\frac{2}{n}+\frac{n}{2}}\leq C (1+s)^{-1}\|u\|_{X(s)}^{1+\frac2{n}},
\end{equation}
and
\begin{equation}
\label{eq:L6norm}
\|u(s,\cdot)\|_{L^{2+\frac4{n}}}^{1+\frac2{n}}\leq C\|\nabla u(s,\cdot)\|_{L^2}\|u(s,\cdot)\|_{L^2}^{\frac2{n}}\leq C (1+s)^{-1-\frac{n}{4}}\|u\|_{X(s)}^{1+\frac2{n}}.
\end{equation}
Thus, we may conclude
\[ \|\partial_x^j G(u)(t,\cdot)\|_{L^2}\leq \|u\|_{X(t)}^{1+\frac2{n}}\int_0^t (1+t-s)^{-\frac{n}{4}-\frac{j}{2}}(1+s)^{-1}\mu(\varepsilon_0(1+s)^{-\frac{n}{2}})\,ds. \]
To estimate $\| G(u)(t,\cdot)\|_{L^\infty}$, the required regularity to the data  increase with $n$, so we split the analysis for $n=1$ and $n=2$. For $n=1$ we may estimate
\[ \| G(u)(t,\cdot)\|_{L^\infty}\leq \int_0^t (1+t-s)^{-\frac{1}{2}}\|h(u(s,\cdot))\|_{L^1\cap L^2}ds, \]
and proceed as before to derive
\[  \|G(u)(t,\cdot)\|_{L^\infty}\leq  \|u\|_{X(t)}^{3}\int_0^t (1+t-s)^{-\frac{1}{2}}(1+s)^{-1}\mu\big(\varepsilon_0(1+s)^{-\frac{1}{2}}\big)\,ds. \]
For $n=2$, applying Lemma \ref{th:Linear} we may estimate
\[ \| G(u)(t,\cdot)\|_{L^\infty}\leq \int_0^t (1+t-s)^{-1}\|h(u(s,\cdot))\|_{L^1\cap H^1}ds. \]
Now, we  have to deal with a new term $\| \nabla h(u(s,\cdot))\|_{ L^2}$. Using \eqref{eq:conditionsmu}, we may estimate
\[|\nabla_x h(u(s,x))|\leq  |u(s,x)| \mu(|u(s,x)|) |\nabla u(s,x)|\]
and
\begin{eqnarray*}
&& \| \nabla h(u(s,\cdot))\|_{ L^2} \lesssim \|u(s,\cdot)\|_{L^{\infty}} \mu(\|u(s,\cdot)\|_{L^{\infty}}) \|\nabla u(s,\cdot)\|_{ L^2}\\
&& \qquad \lesssim (1+s)^{-2}\|u\|_{X(s)}^{2}\mu\big((1+s)^{-1}\|u\|_{X(s)}\big).
\end{eqnarray*}
Therefore
\[  \|G(u)(t,\cdot)\|_{L^\infty}\leq  \|u\|_{X(t)}^{1+\frac2n} \int_0^t (1+t-s)^{-\frac{n}{2}}(1+s)^{-1}\mu\big(\varepsilon_0(1+s)^{-\frac{n}{2}}\big)\,ds, \,\,\, n=1,2. \]
Now, let $\alpha\leq 1$. On the one hand it holds
\[
\int_0^{\frac{t}{2}} (1+t-s)^{-\alpha}(1+s)^{-1}\mu\big(\varepsilon_0(1+s)^{-\frac{n}{2}}\big)\,ds \sim (1+t)^{-\alpha}\int_0^{\frac{t}{2}} (1+s)^{-1}\mu\big(\varepsilon_0(1+s)^{-\frac{n}{2}}\big)\, ds\]
by using $(1+t-s)\sim (1+t)$ on $[0,t/2]$. On the other hand
\begin{eqnarray*}
&& \int_{\frac{t}{2}}^t(1+t-s)^{-\alpha}(1+s)^{-1}\mu(\varepsilon_0(1+s)^{-\frac{n}{2}})\,ds \nonumber \\&& \qquad \lesssim (1+t)^{-\alpha}\int_{\frac{t}{2}}^t (1+t-s)^{-\alpha}(1+s)^{-1+\alpha}\mu(\varepsilon_0(1+s)^{-\frac{n}{2}})\,ds \nonumber \\
 \label{eq:integral2} && \qquad \lesssim(1+t)^{-\alpha}\int_{\frac{t}{2}}^t (1+t-s)^{-1}\mu(\varepsilon_0(1+t-s)^{-\frac{n}{2}})\,ds,
\end{eqnarray*}
where we used $1+s\sim 1+t$ and $1+s\gtrsim 1+t-s$ on $[t/2,t]$. \\
By using the change of variables $r=\frac{1}{\varepsilon_0}(1+s)^{\frac{n}{2}}$, we get
\[ \int_0^{+\infty} (1+s)^{-1}\mu\big(\varepsilon_0(1+s)^{-\frac{n}{2}}\big)\, ds \sim \int_{C_{\varepsilon_0}}^{+\infty} r^{-1}\mu(r^{-1})\, dr,\]
that is finite, due to assumption \eqref{eq:conditionsmu}. Here, $C_{\varepsilon_0}:= \frac{1}{\varepsilon_0}$ tends to $+\infty$ when $\varepsilon_0$ tends to $0$.\\
Summarizing, we arrive at
\begin{equation} \label{eq:mappingdissipation1000critical}
\|Nu\|_{X(t)}
     \lesssim\,C_0\|(\phi,\psi)\|_{\mathcal{A}} + \tilde{C}_{\varepsilon_0}\|u\|_{X(t)}^{1+\frac2{n}},
     \end{equation}
where $\tilde{C}_{\varepsilon_0}$ tends to $0$ for $\varepsilon_0$ to $0$.\\
To derive a Lipschitz condition we recall
\begin{eqnarray*}
&& G u - G v=  \int_0^t  \Phi(t,s,x) \ast_{(x)}\Big(|u|^{1+\frac2n}\mu(|u|) -  |v|^{1+\frac2n} \mu(|v|)\Big)\, ds \\
&& \qquad =\int_0^t  \Phi(t,s,x) \ast_{(x)} \Big(\int_0^1 (d_{|u|}H(|u|)) (v+\tau(u-v)) d\tau\Big)(s,x) (u-v)(s,x) \,ds,
\end{eqnarray*}
where
\[ H : |u| \in \mathbb{R}^+ \to H(|u|)=|u|^{1+\frac2n}\mu(|u|).\]
By using our assumption to $\mu'=\mu'(s)$ we get
\[ \big|d_{|u|}H(|u|)\big| \lesssim |u|^{\frac2n}\mu(|u|).\]
Here we take into consideration that $|u|\leq s_0$ with $s_0$ from \eqref{eq:conditionsmu} for small data solutions.
Applying Minkowski's integral inequality, Theorem \ref{th:Linear} and the monotonicity of $d_{|u|}H(|u|)$ for small $|u|$ gives
\begin{align*}
&\|\nabla_x^{j}(Gu(t,\cdot)-Gv(t,\cdot))\|_{L^2} \\
& \lesssim \!\int_0^t
(1+t-s)^{-\frac{n}{4}-\frac{j}{2}}\Big\|\Big(\int_0^1 \mu(|v+\tau(u-v)|)|v+\tau(u-v)|^{\frac2n} \,d\tau\Big)|u-v|(s,\cdot)\Big\|_{L^1\cap L^2} \,ds\\
& \lesssim \int_0^t\int_0^1 (1+t-s)^{-\frac{n}{4}-\frac{j}{2}} \big\| \big(|u|^{\frac2n}+|v|^{\frac2n}\big)(u-v)(s,\cdot)\big\|_{L^1\cap L^2}\|\mu(|v+\tau(u-v)|)\|_{L^\infty}\,d\tau \,ds.
\end{align*}
By using H\"older's inequality we get
\begin{align*}
&\big\| \big(|u(s,\cdot)|^{\frac2n}+|v(s,\cdot)|^{\frac2n}\big)(u-v)(s,\cdot)\big\|_{L^1}\\
& \qquad \lesssim \big( \|u(s,\cdot)\|_{L^{1+\frac2n}}^{\frac2n}+\|v(s,\cdot)\|_{L^{1+\frac2n}}^{\frac2n}\big) \|(u-v)(s,\cdot)\|_{L^{1+\frac2n}},
\end{align*}
and
\begin{align*}
& \big\| \big(|u(s,\cdot)|^{\frac2n}+|v(s,\cdot)|^{\frac2n}\big)(u-v)(s,\cdot)\big\|_{L^2}\\
& \qquad \lesssim \big( \|u(s,\cdot)\|_{L^{2+\frac4n}}^{\frac2n}+\|v(s,\cdot)\|_{L^{2+\frac4n}}^{\frac2n}\big) \|(u-v)(s,\cdot)\|_{L^{2+\frac4n}}.
\end{align*}
Thus, we can apply Gagliardo-Nirenberg as in \eqref{eq:L3norm} and \eqref{eq:L6norm} to get
\begin{align*}
& \big\| \big(|u(s,\cdot)|^{\frac2n}+|v(s,\cdot)|^{\frac2n}\big)(u-v)(s,\cdot)\big\|_{L^1}\lesssim (1+s)^{-1}\big( \|u\|_{X(s)}^{\frac2n}+\|v\|_{X(s)}^{\frac2n} \big) \|u-v\|_{X(s)},\\
& \big\| \big(|u(s,\cdot)|^{\frac2n}+|v(s,\cdot)|^{\frac2n}\big)(u-v)(s,\cdot)\big\|_{L^2}\lesssim (1+s)^{-1-\frac{n}{4}}\big( \|u\|_{X(s)}^{\frac2n}+\|v\|_{X(s)}^{\frac2n} \big) \|u-v\|_{X(s)}.
\end{align*}
Now we follow the same ideas presented above to conclude
\begin{align*}
& \|\nabla_x^j (Gu(t,\cdot)-Gv(t,\cdot))\|_{L^2} \\ &
\lesssim \|u-v\|_{X(t)}\big( \|u\|_{X(t)}^{\frac2n}+\|v\|_{X(t)}^{\frac2n}\big)\!\int_0^t \int_0^1 (1+t-s)^{-\frac{n}{4}-\frac{j}{2}} (1+s)^{-1}\mu(\|v+\tau(u-v)\|_{L^\infty}) \,d\tau \,ds\\
& \lesssim \|u-v\|_{X(t)}\big( \|u\|_{X(t)}^{\frac2n}+\|v\|_{X(t)}^{\frac2n}\big)\,(1+t)^{-\frac{n}{4}-\frac{j}{2}}\int_0^t \int_0^1 (1+s)^{-1}\mu\big(\varepsilon_0(1+s)^{-\frac{n}{2}}\big)
\,d\tau \,ds \\& \qquad \leq C'_{\varepsilon_0}(1+t)^{-\frac{n}{4}-\frac{j}{2}}\|u-v\|_{X(t)}\big( \|u\|_{X(t)}^{\frac2n}+\|v\|_{X(t)}^{\frac2n}\big),
\end{align*}
where $C'_{\varepsilon_0}$ tends to $0$ for $\varepsilon_0$ to $0$.\\
To estimate $\| Gu(t,\cdot)-Gv(t,\cdot)\|_{L^\infty}$,  we again split the analysis for $n=1$ and $n=2$. For $n=1$ we may proceed as we did to derive the estimates for $\|\nabla_x^j (Gu(t,\cdot)-Gv(t,\cdot))\|_{L^2}$  to conclude
\[ \| Gu(t,\cdot)-Gv(t,\cdot)\|_{L^\infty}   \leq C'_{\varepsilon_0}(1+t)^{-\frac{1}{2}}\|u-v\|_{X(t)}\big( \|u\|_{X(t)}^{2}+\|v\|_{X(t)}^{2}\big),\]
where $C'_{\varepsilon_0}$ tends to $0$ for $\varepsilon_0$ to $0$.\\
For $n=2$, applying Lemma \ref{th:Linear} we may estimate
\begin{align*} & \| Gu(t,\cdot)-Gv(t,\cdot)\|_{L^\infty}\\
& \qquad  \leq \int_0^t
(1+t-s)^{-1}\Big\|\Big(\int_0^1 (d_{|u|}H(|u|)) (v+\tau(u-v)) \,d\tau\Big)(u-v)(s,\cdot) \Big\|_{L^1\cap H^1} \,ds.
\end{align*}
The only new term to be considered is
\[\big\| (d_{|u|}H(|u|)) (v+\tau(u-v))(s,\cdot) (u-v)(s,\cdot) \big\|_{\dot H^1}.\]
Using \eqref{eq:conditionsmu}, we may estimate
\[ \big|\nabla_x d_{|u|}H(|u|)(v+\tau(u-v))\big|\lesssim  (|\nabla u| + |\nabla v|) \mu(|v+\tau(u-v)|)\]
and
\begin{align*}
& \big\| (d_{|u|}H(|u|)) (v+\tau(u-v))(s,\cdot) (u-v)(s,\cdot) \big\|_{\dot H^1}\\
& \qquad \lesssim  \mu(\|v+\tau(u-v)\|_{L^\infty})
 \big(\|\nabla u(s,\cdot)\|_{L^2} + \|\nabla v(s,\cdot)\|_{L^2}\big) \|(u-v)(s,\cdot)\|_{L^\infty} \\
& \qquad \quad +\mu(\|v+\tau(u-v)\|_{L^\infty})(\|u(s,\cdot)\|_{L^\infty}+ \|v(s,\cdot)\|_{L^\infty})\|\nabla(u-v)(s,\cdot)\|_{L^2}\\
&\qquad \lesssim  (1+s)^{-2}\mu\big(\varepsilon_0(1+s)^{-1}\big)( \|u\|_{X(s)}+\|v\|_{X(s)}) \|u-v\|_{X(s)}.
\end{align*}
Hence, we may estimate
\begin{align*} & \|Gu(t,\cdot)-Gv(t,\cdot)\|_{L^\infty} \\
& \qquad \lesssim ( \|u\|_{X(t)}+\|v\|_{X(t)}) \|u-v\|_{X(t)}\int_0^t
(1+t-s)^{-1}(1+s)^{-1}\mu\big(\varepsilon_0(1+s)^{-1}\big) \,ds\\
& \qquad \leq C'_{\varepsilon_0} (1+t)^{-1} ( \|u\|_{X(t)}+\|v\|_{X(t)}) \|u-v\|_{X(t)},
\end{align*}
where $C'_{\varepsilon_0}$ tends to $0$ for $\varepsilon_0$ to $0$.\\
Summarizing all the estimates implies
     \begin{equation}
\label{eq:contractiondissipation100critical}
\|Nu-Nv\|_{X(t)}
     \leq C_{\varepsilon_0}\|u-v\|_{X(t)}\big( \|u\|_{X(t)}^{\frac2n}+\|v\|_{X(t)}^{\frac2n}\big)
\end{equation}
for any~$u,v\in X(t)$, where $C_{\varepsilon_0}$ tends to $0$ for $\varepsilon_0$ to $0$.
Due to \eqref{eq:mappingdissipation1000critical} the operator $N$ maps~$X(t)$ into itself if $\varepsilon_0$ is small enough.
The existence of a unique global (in time) Sobolev solution $u$ follows by contraction \eqref{eq:contractiondissipation100critical} and continuation argument for small data.
\end{proof}

\section{Non-existence result  via test function method}

   Following the proof of Theorem \ref{GEDW}, we obtain a local (in time) Sobolev solution $u \in C\big([0, T), H^1(\R^n)\cap L^{\infty}(\R^n)\big)$ to \eqref{eq:blowupdampedwave}. For this reason we restrict ourselves to prove that this solution can not exist globally in time.
  \begin{proof}[Theorem \ref{resultblowupdampedwave}] \\
 We introduce the following functions:
\[ \eta(s)=\begin{cases} 1 & \text{ if } s\in [0,\frac{1}{2}],\\
\text{ is decreasing} & \text{ if } s\in (\frac{1}{2},1), \\
0 & \text{ if } s\in [1, \infty), \end{cases} \qquad  \eta^*(s)=\begin{cases} 0 & \text{ if } s\in [0,\frac{1}{2}],\\
\eta(s) & \text{ if } s\in [\frac{1}{2}, \infty), \end{cases}\]
where the function $\eta=\eta(s)$ is supposed to belong to $C^\infty[0,\infty)$. For $R\geq R_0>0$, where $R_0$ is a large parameter, we define for $ (t,x)\in [0,\infty) \times \R^n$ the cut-off functions
\[ \psi_R=\psi_R(t,x)=\eta\bigg(\frac{|x|^2+t}{R} \bigg)^{n+2}\,\,\,\mbox{and}\,\,\, \psi^*_R=\psi^*_R(t,x)=\eta^*\bigg(\frac{|x|^2+t}{R} \bigg)^{n+2}. \]
We note that the support of $ \psi_R$ is contained in \[ Q_R= [0, R] \times B_{\sqrt{R}}\,\,\,\mbox{ with}\,\,\,
 B_{\sqrt{R}}= \big\{ x\in  \R^n: \ |x|\leq \sqrt{R}\big\}.\] The support of $ \psi^*_R$ is contained in
 \[ Q_R^*=Q_R \setminus  \Big\{(t,x): |x|^2+t\leq \frac{R}2 \Big\}.\] We suppose that the Sobolev solution $u=(t,x)$ exists globally in time, that is, the lifespan is $T=T(u)=\infty$.  We define the functional
\[ I_R=\int_{Q_R} h(|u(t, x)|)\psi_R(t,x)\,d(t,x) \,\,\,\mbox{with}\,\,\ h(s):=s^{1+\frac2n}\mu(s).\]
Then, by equation \eqref{eq:blowupdampedwave}, after using integration by parts we arrive at
\[  I_R = -\int_{\R^n} g(x) \psi_R(0,x)\, dx + \int_{Q_R} u(t, x)\big( \partial_t^2\psi_R-\Delta \psi_R-\partial_t \psi_R\big) \,d(t,x).\]
It holds
\begin{align*}
\partial_t \psi_R &= \frac{n+2}{R}  \eta\bigg(\frac{|x|^2+t}{R} \bigg)^{n+1}  \eta'\bigg(\frac{|x|^2+t}{R} \bigg);\\
\partial_t^2 \psi_R &=\frac{(n+2)(n+1)}{R^2}  \eta\bigg(\frac{|x|^2+t}{R} \bigg)^n \eta'\bigg(\frac{|x|^2+t}{R} \bigg)^2\\
& \quad + \frac{n+2}{R^2}\eta\bigg(\frac{|x|^2+t}{R} \bigg)^{n+1}  \eta''\bigg(\frac{|x|^2+t}{R} \bigg);\\
\partial_{x_j}^2\psi_R &= \frac{4(n+2)(n+1)x_j^2}{R^2}  \eta\bigg(\frac{|x|^2+t}{R} \bigg)^n \eta'\bigg(\frac{x^2+t}{R} \bigg)^2\\
&\quad + \frac{4(n+2)x_j^2}{R^2}\eta\bigg(\frac{|x|^2+t}{R} \bigg)^{n+1} \eta''\bigg(\frac{|x|^2+t}{R} \bigg)\\
& \quad +\frac{2(n+2)}{R} \eta\bigg(\frac{|x|^2+t}{R} \bigg)^{n+1} \eta'\bigg(\frac{|x|^2+t}{R} \bigg).
\end{align*}
Thus, since $0\leq \eta\leq 1$ and $\eta',\, \eta''$ are bounded on $[0,\infty)$, there exists $C>0$ such that for each $(t,x)\in \mbox{supp}\, \psi_R$ it holds
\[ \big|\partial_t^2 \psi_R-\Delta \psi_R-\partial_t \psi_R\big|\leq \frac{C}{R}(\psi^*_R(t,x))^{\frac{n}{n+2}}.\]
Thus, we get
\begin{align}
\label{eq:EstimateY(R)}
I_R&=\int_{Q_R} h(|u(t, x)|)\psi_R(t,x)\,d(t,x) \leq -\int_{\R^n} g(x)\psi_R(0,x) \,dx \nonumber\\
& \quad + \frac{C}{R}\int_{Q_R} |u(t,x)|(\psi^*_R(t,x))^{\frac{n}{n+2}} \,d(t,x).
\end{align}
By applying Lemma \ref{Lemma:Jensen} with $\alpha\equiv 1$ we get
\[ h\left(\frac{\int_{Q_R^*} |u(t, x)|(\psi^*_R(t,x))^{\frac{n}{n+2}} \,d(t,x) }{\int_{Q_R^*}1 \,d(t,x)}\right)\leq \frac{\int_{Q_R^*} h\big(|u(t,x)|(\psi^*_R(t,x))^{\frac{n}{n+2}}\big) \,d(t,x)}{\int_{Q_R^*}1\, d(t,x)}.\]
Taking account of
\begin{eqnarray*}
&& \int_{Q_R^*} |u(t, x)|(\psi^*_R(t,x))^{\frac{n}{n+2}} \,d(t,x) = \int_{Q_R} |u(t, x)|(\psi^*_R(t,x))^{\frac{n}{n+2}} \,d(t,x),\\
&&  \int_{Q_R^*}1\, d(t,x) = C \int_{Q_R}1\, d(t,x),
\end{eqnarray*}
we arrive at the estimate
\[ h\left(\frac{\int_{Q_R} |u(t, x)|(\psi^*_R(t,x))^{\frac{n}{n+2}} \,d(t,x) }{C \int_{Q_R}1 \,d(t,x)}\right)\leq \frac{\int_{Q_R} h\big(|u(t,x)|(\psi^*_R(t,x))^{\frac{n}{n+2}}\big) \,d(t,x)}{C \int_{Q_R}1\, d(t,x)}.\]
Notice that, since the modulus of continuity $\mu$ is non-decreasing, we can estimate $$ h\big(|u(t,x)|(\psi^*_R(t,x))^{\frac{n}{n+2}}\big)\leq h(|u(t,x)|)\psi^*_R(t,x).$$
Moreover, $$ \int_{Q_R} 1\, d(t,x)=R^{\frac{n+2}{2}}.$$
Thus, thanks again to $\mu$ to be a non-decreasing function, there exists  $h^{-1}$ and  we may conclude
\begin{equation}
\label{eq:JensenY(R)}
\int_{Q_R} |u(t, x)|(\psi^*_R(t,x))^{\frac{n}{n+2}} \,d(t,x) \leq CR^{\frac{n+2}{2}}h^{-1}\left(\frac{\int_{Q_R} h(|u(t, x)|)\psi^*_R(t,x) \,d(t,x)}{CR^{\frac{n+2}{2}}}\right).
\end{equation}
Let us define the functions
\[ y=y(r)= \int_{Q_R} h(|u(t,x)|)\psi^*_r(t,x) \,d(t,x) \,\,\,\mbox{and}\,\,\,Y= Y(R)=\int_0^R y(r)r^{-1}\,dr. \]
Then, it holds
\begin{align*}
& Y(R)= \int_0^R\bigg(\int_{Q_R} h(|u(t, x)|)\psi^*_r(t,x) \,d(t,x)\bigg)r^{-1}\,dr\\ & \qquad = \int_{Q_R} h(|u(t, x)|) \bigg( \int_0^R \eta^*\bigg(\frac{|x|^2+t}{r}\bigg)^{n+2}\,r^{-1}\,dr\bigg)\,d(t,x) \\& \qquad =
\int_{Q_R} h(|u(t, x)|) \bigg( \int_{\frac{|x|^2+t}{R}}^\infty (\eta^*(s))^{n+2}s^{-1}ds\bigg)\,d(t,x).
\end{align*}
Since $\text{supp}\,\eta^*\subset [1/2,1]$ and $\eta^*$ is a non-increasing function on its support, we obtain the estimate
\[ \int_{\frac{|x|^2+t}{R}}^\infty (\eta^*(s))^{n+2}s^{-1}\,ds\leq \eta\bigg(\frac{|x|^2+t}{R}\bigg)^{n+2} \int_{\frac{1}{2}}^1 s^{-1}\,ds \leq \log(2) \eta\bigg(\frac{x^2+t}{R}\bigg)^{n+2}.\]
Consequently, we may conclude
\[ Y(R) \leq \log(2) \int_{Q_R} h(|u(t, x)|) \psi_R(t,x)\,d(t,x)=\log(2)\, I_R.\]
Moreover, we notice
\[ Y'(R)R=y(R)= \int_{Q_R} h(|u(t, x)|)\psi^*_R(t,x) \,d(t,x).\]
Thus, by \eqref{eq:EstimateY(R)} and \eqref{eq:JensenY(R)}, we get
\[ \frac{Y(R)}{\log(2)}\leq C^2 R^{\frac{n}{2}}h^{-1}\bigg(\frac{Y'(R)}{CR^{\frac{n}{2}}}\bigg).\]
It follows
\[ h\bigg(\frac{Y(R)}{C^2\log(2)R^{\frac{n}{2}}}\bigg)\leq \frac{Y'(R)}{CR^{\frac{n}{2}}}.\]
Thus, we have
\[\bigg(\frac{Y(R)}{C^2\log(2)R^{\frac{n}{2}}}\bigg)^{\frac{n+2}{n}}\mu\bigg(\frac{Y(R)}{C^2\log(2)R^{\frac{n}{2}}}\bigg)\leq \frac{Y'(R)}{CR^{\frac{n}{2}}}.\]
For each $R\geq R_0$, since $Y=Y(r)$ is increasing we have $Y(R)\geq Y(R_0)$. Thus, since $\mu$ is non-decreasing, we have
\[\bigg(\frac{Y(R)}{C^2\log(2)R^{\frac{n}{2}}}\bigg)^{\frac{n+2}{n}}\mu\bigg(\frac{Y(R_0)}{C^2\log(2)R^{\frac{n}{2}}}\bigg)\leq \frac{Y'(R)}{CR^{\frac{n}{2}}}.\]
Thus, we have
\[\frac{1}{R(C^2\log(2))^{\frac{n+2}{n}}}\mu\bigg(\frac{Y(R_0)}{C^2\log(2)R^{\frac{n}{2}}}\bigg)\leq \frac{Y'(R)}{CY(R)^{\frac{n+2}{n}}}.\]
By integrating from $R_0$ to $R$, we can conclude that there exist  constants $c_1$, $c_2$ such that
\begin{equation}
\label{eq:blowupestimate}
\int_{R_0}^{R} \frac{1}{s}\,\mu(c_2s^{-\frac{n}{2}})\,ds=c_1\int_{R_0^ \frac{n}{2}}^{R^\frac{n}{2}} \frac{1}{s}\,\mu(c_2s^{-1})\,ds\leq \bigg[-\frac{1}{Y(s)^{\frac{2}{n}}}\bigg]_{R_0^\frac{n}{2}}^{R^\frac{n}{2}}\leq \frac{1}{Y(R_0^\frac{n}{2})^{\frac{2}{n}}}.
\end{equation}
Due to the assumption that $u=u(t, x)$ exists globally in time it is allowed to form the limit $R\rightarrow \infty$ in \eqref{eq:blowupestimate}. But this produces a contradiction, due to the
fact that the right-hand side is bounded and the modulus of continuity $\mu$ satisfies condition \eqref{eq:blowupcondition}.
This completes our proof.
\end{proof}
\section{Appendix} \label{SecAppendix}
In the Appendix we include the following generalized version of Jensen Inequality (\cite{PKJF}).
\begin{lem}
\label{Lemma:Jensen}
Let $\Phi$ be a convex function on $\R$. Let $\alpha=\alpha(x)$ be defined and non-negative almost everywhere on $\Omega$, such that $\alpha$ is positive in a set of positive measure. Then, it holds
\[ \Phi\bigg(\frac{\int_\Omega u(x)\alpha(x)\,dx}{\int_\Omega \alpha(x)\,dx}\bigg)\leq \frac{\int_\Omega \Phi(u(x))\alpha(x)\,dx}{\int_\Omega \alpha(x)\,dx} \]
for all non-negative functions $u$ provided that all the integral terms are meaningful.
\end{lem}
\begin{proof}
Let $\gamma>0$ be fixed. From the convexity of $\Phi$ it follows that there exists $k\in \R^1$, such that
\[ \Phi(t)-\Phi(\gamma)\geq k(t-\gamma )  \,\,\, \mbox{ for all } t\geq 0.\]
Putting $t=u(x)$ and multiplying the last inequality by $\alpha(x)$, we get after integration over $\Omega$ that
\[ \int_\Omega \Phi(u(x))\alpha(x)\,dx- \Phi(\gamma)\int_{\Omega}\alpha(x)\,dx\geq k\Big( \int_\Omega u(x)\alpha(x)\,dx-\gamma \int_\Omega \alpha(x)\,dx \Big).\]
The statement follows by putting
\[ \gamma=\frac{\int_\Omega u(x)\alpha(x)\,dx}{\int_\Omega\alpha(x)\,dx}.\]
\end{proof}

\begin{acknowledgements}
The discussions on this paper began during the time the third author spent a two weeks research stay in November 2018 at the Department of Mathematics and Computer Science of University of S\~{a}o Paulo, FFCLRP. The stay of the third author was  supported by Funda\c{c}\~{a}o de Amparo \`{a} Pesquisa do Estado de S\~{a}o Paulo (FAPESP), grant 2018/10231-3.
The second author contributed to this paper during a four months stay within Erasmus+ exchange program during the period October 2018 to February 2019.
The first author have been partially supported by FAPESP,
grant number 2017/19497-3.
\end{acknowledgements}



\end{document}